\documentclass[12pt,letterpaper]{amsart}

\usepackage{color}
\usepackage{tikz-cd}
\usetikzlibrary{arrows.meta, positioning}
\usepackage{esint,amssymb}
\usepackage{graphicx}
\usepackage{MnSymbol}
\usepackage{mathtools}
\usepackage[colorlinks=true, pdfstartview=FitV, linkcolor=blue, citecolor=pink, urlcolor=blue,pagebackref=false]{hyperref}
\usepackage{microtype}
\usepackage{amsmath}
\usepackage{xifthen}
\usepackage{verbatim}
\definecolor{darkgreen}{rgb}{0,0.5,0}
\definecolor{darkblue}{rgb}{0,0,0.7}
\definecolor{darkred}{rgb}{0.9,0.1,0.1}
\usepackage[all]{xy}
\usepackage[makeroom]{cancel}
\usepackage{enumitem}
\makeatletter
\def\@tocline#1#2#3#4#5#6#7{\relax
  \ifnum #1>\c@tocdepth 
  \else
    \par \addpenalty\@secpenalty\addvspace{#2}%
    \begingroup \hyphenpenalty\@M
    \@ifempty{#4}{%
      \@tempdima\csname r@tocindent\number#1\endcsname\relax
    }{%
      \@tempdima#4\relax
    }%
    \parindent\z@ \leftskip#3\relax \advance\leftskip\@tempdima\relax
    \rightskip\@pnumwidth plus4em \parfillskip-\@pnumwidth
    #5\leavevmode\hskip-\@tempdima
      \ifcase #1
       \or\or \hskip 1em \or \hskip 2em \else \hskip 3em \fi%
      #6\nobreak\relax
    \dotfill\hbox to\@pnumwidth{\@tocpagenum{#7}}\par
    \nobreak
    \endgroup
  \fi}
\makeatother

\newtheorem{theorem}{Theorem}
\newtheorem{proposition}[theorem]{Proposition}
\newtheorem{lemma}[theorem]{Lemma}
\newtheorem{corollary}[theorem]{Corollary}

\theoremstyle{definition}

\newcommand{\cref}[1]{Corollary~\ref{c.#1}}

\numberwithin{equation}{section}
\numberwithin{theorem}{section}

\newcommand{\R}{\mathbb{R}}

\renewcommand{\subset}{\subseteq}

\newcommand{\eps}{\varepsilon}

\newcommand{\test}[1][]{%
\ifthenelse{\equal{#1}{}}{omitted}{given}%
}

\renewcommand{\tilde}{\widetilde}

\renewcommand{\part}{\partial}

\usepackage{dsfont}
\usepackage{slashed}

\newcommand{\ad}{\operatorname{ad}}

\begin{document}
\title{Off-diagonal bounds for spectral projector kernels}

\author[Panagiotis Dimakis]{Panagiotis Dimakis}
\subjclass[2010]{}
\date{\today}
\begin{abstract}
We study the growth of the spectral function of the Laplace-Beltrami operator on a Riemannian manifold $M$ of dimension $n$. For $n\geq 3$ we show that for any $x\in M$ there is a full measure subset $Z_x\subset M$ such that for all $y\in Z_x$ $$E_{\Lambda}(x,y) = O_{M,x,y,\eps}(\Lambda^{n/2}(\log\Lambda)^{3/2}(\log\log\Lambda)^{1/2+\eps}). $$ For $n=2$ we prove a stronger statement. For any $x\in M$ there is a full measure subset $Z_x\subset M$ such that for all $y\in Z_x$ $$E_{\Lambda}(x,y) = O_{M,x,y,\eps}(\Lambda^{5/6}(\log\Lambda)^{1/6}(\log\log\Lambda)^{1/6+\eps}) .$$
\end{abstract}

\maketitle

\tableofcontents

\section{Introduction}

Let $(M^n,g)$ be an $n-$dimensional smooth, closed, connected Riemannian manifold. Let $\Delta$ denote the non-negative Laplace-Beltrami operator and write 
\begin{equation*}
\Delta\phi_j = \lambda_j^2\phi_j, ~ 0=\lambda_0\le \lambda_1\le \lambda_2\le\dots
\end{equation*}
for an orthonormal basis of eigenfunctions $\{\phi_j(x)\}$. The integral kernel of the spectral projection operator 
\[
\mathbf 1_{(0,\Lambda)}\sqrt{\Delta}
\]
is given by the spectral function 
\begin{equation*}
E_{\Lambda}(x,y) = \sum\limits_{0<\lambda_j<\Lambda}\phi_j(x)\overline{\phi_j(y)}.
\end{equation*}
By the Weyl formulae \cite{Levitan1953,Avakumovic1956, Hormander1968}, 
\begin{equation*}
\begin{split}
E_{\Lambda}(x,x) &=   (2\pi)^{-n}\mathrm{vol}_{\R^n}(B)\mathrm{vol}_g(M)\Lambda^n + O(\Lambda^{n-1})\\
E_{\Lambda}(x,y) &=  O(\Lambda^{n-1}),
\end{split}
\end{equation*}
for $x\neq y$. These bounds are sharp and attained on a round $n$-sphere. Assuming that the geodesic loops at $x$ and respectively the geodesics joining $x$ and $y$ are of measure zero these bounds can be improved to $o(\Lambda^{n-1})$ \cite{Guillemin1975}. A plethora of refinements along this direction have been obtained over the years, see for example \cite{Petridis2002, Canzani2015, Canzani2018, Canzani2021} and references therein. 

Lower bounds under non-conjugacy along any shortest
geodesic segment joining $x,y$ were obtained in \cite{Polterovich2007}. Specifically, for such pairs 
\[
E_{\Lambda}(x,y) \neq o(\Lambda^{(n-1)/2}). 
\]
In \cite{Polterovich2009} the authors study the average growth of the spectral function and conjecture that the typical rate of growth should be $O(\Lambda^{(n-1)/2})$. The first main theorem in this paper shows that generically the rate of growth is $O(\Lambda^{n/2}\log(\Lambda)^{3/2+\eps})$, only $1/2$ away from the conjectured typical growth. Specifically, we prove

\begin{theorem}\label{M1}
Given any point $x\in M$ there exists a set $Z_x\subset M$ of full measure such that for every $y\in Z_x$ and for every $\eps>0$ the bound 
\begin{equation}\label{E1}
E_{\Lambda}(x,y) = O_{M,x,y,\eps}(\Lambda^{n/2}(\log\Lambda)^{3/2}(\log\log\Lambda)^{1/2+\eps})
\end{equation}
holds for \textbf{every} sufficiently large real $\Lambda$. 
\end{theorem}
This is a genuine power improvement over the general off-diagonal bounds $o(\Lambda^{n-1})$ mentioned above when $n\geq 3$. For $n=2$  we prove the following 
\begin{theorem}\label{M2}
Given any point $x\in M$ there exists a set $Z_x\subset M$ of full measure such that for every $y\in Z_x$ and for every $\eps>0$ the bound 
\begin{equation}\label{E2}
E_{\Lambda}(x,y) = O_{M,x,y,\eps}(\Lambda^{5/6}(\log\Lambda)^{1/6}(\log\log\Lambda)^{1/6+\eps})
\end{equation}
holds for \textbf{every} sufficiently large real $\Lambda$. 
\end{theorem}

The main idea behind the proof is the existence of a maximal dyadic estimate of the form 
\begin{equation*}
\int_M \sup\limits_{0<\Lambda\le 2^m} |E_{\Lambda}(x,y)|^2\,dy \le C_M m^22^{mn}.
\end{equation*}
Once such an estimate is established, it can be combined with the Chebyshev inequality and the classical Borel-Cantelli Lemma to produce a subset $Z_x\subset M$ of full measure such that for any given $y\in Z_x$ there is a threshold $\Lambda_0$ depending on $y$ such that \eqref{E1} holds for $\Lambda_0\le \Lambda$. 

Surprisingly, for $n\geq 3$ the above maximal dyadic estimate is an immediate consequence of the Rademacher-Menshov inequality. 

The $n=2$ case is less straightforward. The idea is to prove an $L^p$-based maximal dyadic estimate since this can, and in fact does give a better exponent when one applies the Chebyshev inequality. The proof is technically involved and uses the spectral cluster estimates of Sogge \cite{Sogge1993} and the Christ-Kiselev theorem \cite{Christ2001} which are stated precisely in section \ref{prelim}. It is thanks to the latter theorem that a power of logarithm is absent from \eqref{E2} while present at \eqref{E1}. 

\subsection*{AI Disclosure} I got the idea of using maximal estimates for bounding the spectral function from \cite{CD2026}. The main idea in \cite{CD2026} was obtained through extensive interactions with ChatGPT-$5.6$ Sol. In this paper, I used ChatGPT-5.6 Sol to assist with part of the proof of Theorem \ref{M2}. Maybe the most meaningful recommendation is the introduction of the cutoff function in Lemma \ref{aux2} which completely avoids the issue of having to estimate the low-frequency sum directly. \cite{Polterovich2009} seem to have a way of dealing with this by splitting $u_k$ into two pieces and estimating them separately. 

\subsection*{Acknowledgements} I would like to thank Dimitrios Chatzakos who, late in January $2026$, shared with me his passion for the hyperbolic lattice problem. It is that one conversation (and quite obviously also the ones that followed) that led me to the simple and, in my admittedly biased view, beautiful idea behind the proofs of this paper.

\section{Preliminaries}\label{prelim}

\subsection{Surface spectral cluster estimates}

Define 
\[
Q_k := \sum\limits_{k\le \lambda_j<k+1} \phi_j(x)\overline{\phi_j(y)}. 
\]
The following proposition is corolary $5.1.2$ in \cite{Sogge1993},
\begin{proposition}\label{Sogge}
\begin{equation*}
\|Q_k f\|_{L^p(M)} \le C_{M,p}(1+k)^{\delta(p)}\|f\|_{L^2(M)}
\end{equation*}
where 
\[
\delta(p) = \left\{ \begin{aligned}
    &\frac{1}{4} - \frac{1}{2p},&& \text{for } 2\le p\le6 \\
    &\frac{1}{2} - \frac{2}{p}, && \text{for } 6\le p <\infty
\end{aligned} \right.
\]
\end{proposition}

\subsection{Finite maximal inequalities} 

\begin{theorem}[A finite Rademacher-Menshov inequality]\label{RME}
Let $u_1,\dots,u_N$ be pairwise orthogonal elements of $L^2(M,\mu)$. Then 
\begin{equation*}
\int_M \max\limits_{1\le k\le N} \left|\sum\limits_{j=1}^ku_j(z) \right|^2\,d\mu(z) \le C(\log(N))^2\sum\limits_{j=1}^N\|u_j\|_{L^2(M)}^2
\end{equation*}
where $C$ is an absolute constant. 
\end{theorem}
This is a well-known inequality. A modern reference is \cite{Meaney2007}. The next theorem of Christ and Kiselev \cite{Christ2001} should be thought of as an analog of the Rademacher-Menshov inequality for $p>2$. The stronger intitial boundedness assumption allows one to eliminate the logarithm which is present in the Rademacher-Menshov inequality \ref{RME}. 
\begin{theorem}\label{CKE}
Let $2<p<\infty$. Assuming that 
\begin{equation*}
T:\ell^2(\{1,...,N\}) \to L^p(M,\mu)
\end{equation*}
is a bounded linear operator, then for all $c=(c_1,...,c_N)$ 
\begin{equation*}
\|\max\limits_{0\le k\le N}|T(c_1,\dots,c_k,0,\dots,0)|\|_p \le C_p\|T\|_{\ell^2\to L^p}\|c\|_{\ell^2}
\end{equation*}
where $C_p = (1-2^{-(1/2-1/p)})^{-1}$. 
\end{theorem}
\begin{corollary}\label{CKC}
Let $2<p<\infty$ and $k\geq 2$. Let $\psi_1,...,\psi_N$ be a set of orthonormal eigenfunctions with frequencies in $[k,k+1)$. Then 
\begin{equation*}
\left\|\max\limits_{0\le r\le N} \left|\sum\limits_{v =1}^r c_v\psi_v \right|\right\|_p \le C_{M,p}k^{\delta(p)} \left(\sum\limits_{v=1}^N |c_v|^2\right)^{1/2}. 
\end{equation*}
\begin{proof}
Define 
\[
T_k(c) = \sum\limits_{v=1}^N c_v\psi_v.
\]
The image of this operator is in the range $\mathrm{ran} Q_k$ and orthonormality gives 
\[
\|T_kc\|_2^2 = \sum\limits_{v=1}^N|c_v|^2. 
\]
Apply \ref{Sogge} to get 
\[
\|T_kc\|_p \le C_{M,p}k^{\delta(p)}\|c\|_{\ell^2}
\]
and thus $\|T_k\|_{\ell^2\to L^p} \le C_{M,p}k^{\delta(p)}$ which immediately implies the result. 
\end{proof}
\end{corollary}

\section{Proof of Theorem \ref{M1}}

This is deceptively easy. The local Weyl law gives 
\[
\sum\limits_{0<\lambda_j<L} |\phi_j(x)|^2 \le C_M(1+L)^n
\]
with the constant uniform in $x\in M$. Let 
\[
J_m = \{j: 0<\lambda_j<L\}.
\]
For fixed $x$ the functions $u_j(y) = \phi_j(x)\overline{\phi_j(y)}$ for $j\in J_m$ are pairwise orthogonal in $L^2(M,\,dy)$ and $\|u_j\|_{L^2(y)} = |\phi_j(x)|^2$. Applying the Rademacher-Menshov inequality \ref{RME} to this finite set of functions and using the local Weyl law we obtain 
\begin{equation}\label{mainest1}
\int_M \sup\limits_{0<\Lambda\le 2^m} |E_{\Lambda}(x,y)|^2\,dy \le C_M m^22^{mn}. 
\end{equation}
Now define $b_m = m^{1/2}(\log m )^{1/2+\eps}$ and
\[
\mathcal E_m = \{y: \sup\limits_{0<\Lambda\le 2^m} |E_{\Lambda}(x,y)| > 2^{mn/2}mb_m\}.
\]
The estimate \eqref{mainest1} and the Chebyshev inequality give
\[
\mu(\mathcal E_m) \le C_Mb_m^{-2} = \frac{C_M}{m(\log m)^{1+2\eps}}. 
\]
This in particular implies that 
\[
\sum\limits_{m\geq 0} \mu(\mathcal E_m) <\infty,
\]
so the first Borel-Cantelli Lemma implies that 
\[
\mu(\limsup\limits_{m\to\infty} \mathcal E_m) = 0.
\]
Define $Z_x = M\backslash \limsup\limits_{m\to\infty} \mathcal E_m$. Then for any $y\in Z_x$ there exists some $m_0$ such that for all $m\geq m_0$ and $\Lambda\le 2^m$
\[
|E_{\Lambda}(x,y)| \le 2^{mn/2}mb_m. 
\]
This immediately gives 
\[
E_{\Lambda}(x,y) = O_{M,x,y,\eps}(\Lambda^{n/2}(\log\Lambda)^{3/2}(\log\log\Lambda)^{1/2+\eps}).
\]

\section{Proof of Theorem \ref{M2}}

\subsection{Auxiliary results}

\begin{lemma}\label{aux}
If $\chi\in C^{\infty}(M)$ then for all $N\geq 0$ 
\[
\|Q_m \chi Q_n\|_{L^2(M)\to L^2(M)} \le C_{N,\chi} (1+|m-n|)^{-N}
\]
for $m,n\geq 0$.
\end{lemma}
%


%
\begin{proof}
Let $A:= \sqrt{\Delta}$. This is a classical pseudodifferential operator of order one. Symbolic calculus shows that $B_N = \ad_A^N(\chi)$ where $\ad_A(B) = [A,B]$ is of order zero for every $N$ and hence bounded in $L^2(M)$. Assume first that $m\ge n+2$. Let $A_m = A|_{\mathrm{ran}Q_m}$ and $A_n = A|_{\mathrm{ran}Q_n}$. These are self-adjoint operators on finite dimensional spaces with spectra contained in $[m,m+1)$ and $[n,n+1)$ respectively. On operators $T:\mathrm{ran}Q_n\to \mathrm{ran}Q_m$ define 
\[
\mathcal L(T) = A_mT-TA_n.
\]
Let $g = m-n-1>0$. For any such operator $S$ the integral 
\[
\mathcal R(S) = \int_0^{\infty} e^{-tA_m}Se^{tA_n}\,dt 
\]
converges in operator norm since $\|e^{-tA_m}Se^{tA_n}\| \le e^{-tg}\|S\|$ and $\|\mathcal R(S)\| \le g^{-1}\|S\|$. It is easy to verify that $\mathcal R$ is the inverse of $\mathcal L$ and therefore $\|\mathcal L^{-1}\|\le g^{-1}$. 

Since $Q_m,Q_n$ commute with $A$, $\mathcal L(Q_m\chi Q_n) = Q_m[A,\chi]Q_n$ and inductively $\mathcal L^N(Q_m\chi Q_n) = Q_mB_NQ_n$. This implies that 
\[
\|Q_m\chi Q_n\|_{L^2(M)\to L^2(M)} \le g^{-N}\|Q_mB_NQ_n\|_{L^2(M)\to L^2(M)}\le C_{N,\chi}(m-n-1)^{-N}.
\]
For $n\geq m+2$ repeat the above argument for $(Q_m \chi Q_n)^{\star} = Q_n \bar\chi Q_m$. If $|m-n|\le 1$ then the result follows immediately from 
\[
\|Q_m\chi Q_n\|_{L^2(M)\to L^2(M)}\le \|\chi\|_{\infty}.
\]
\end{proof}

We fix $\chi\in C^{\infty}(M)$ vanishing in a neighborhood of $x$. Let $u_{\Lambda}(y) = E_{\Lambda}(x,y) = E_{\Lambda}\delta_x$ and $\tilde u_{\Lambda}(y) = u_{\Lambda}(y) + \mathrm{vol}(M)^{-1} =: \mathbf{1}_{[0,\Lambda)}\sqrt{\Delta} \delta_x$. The next Lemma is stated for $\tilde u_{\Lambda}(y)$ so that the crucial identity \eqref{didentity} holds. 
\begin{lemma}\label{aux2}
For integers $k\ge 2$, $m\ge 0$ and $N\geq 1$,
\[
\|Q_m\chi \tilde u_k\|_{L^2(M)} \le C_{M,\chi,x}(1+k)^{1/2}(1+|m-k|)^{-N}.
\]
\end{lemma}
\begin{proof}
If $m\ge k$ then 
\begin{equation*}
\begin{split}
\|Q_m\chi \tilde u_k\|_{L^2(M)} &\le \sum\limits_{n=0}^{k-1}\|Q_m\chi Q_n\|_{L^2(M)\to L^2(M)} \|Q_n\delta_x\|_{L^2(M)}\\
&\le C_{N_1,\chi} \sum\limits_{n=0}^{k-1}(1+|m-n|)^{-N_1}(1+n)^{1/2}\\
&\le C_{N_1,\chi}(1+|m-k|)^{1-N_1}(1+k)^{1/2},
\end{split}
\end{equation*}
where the inequality in the second line was obtained by applying Lemma \ref{aux} and the local Weyl law. To conclude, choose $N_1 = 1+N$.

For $m<k$ use the identity 
\begin{equation}\label{didentity}
\sum\limits_{n=0}^{\infty} Q_m\chi Q_n \delta_x = 0
\end{equation}
in $L^2(M)$ to write 
\[
Q_m(\chi \tilde u_k) = -  \sum\limits_{n=k}^{\infty} Q_m\chi Q_n \delta_x.
\]
Thus, using Lemma \ref{aux} and the local Weyl law we have 
\begin{equation*}
\begin{split}
\|Q_m(\chi \tilde u_k)\|_{L^2(M)} &\le \sum\limits_{n=k}^{\infty}\|Q_m\chi Q_n\|_{L^2(M)\to L^2(M)} \|Q_n\delta_x\|_{L^2(M)}\\
&\le C_{N_2,\chi} \sum\limits_{n=k}^{\infty} (1+n-m)^{-N_2}(1+n)^{1/2}\\
&\le C_{N_2,\chi} (1+k)^{1/2}(1+|m-k|)^{3/2-N_2}.
\end{split}
\end{equation*}
To conclude choose $N_2 = 3/2+ N$. 
\end{proof}
\begin{proposition}\label{aux3}
It holds that
\[
\|\chi u_k\|_{L^p(M)} \le C_{M,p,\chi,x}k^{1/2+\delta(p)}. 
\]
\end{proposition}
\begin{proof}
First, we observe that 
\[
\|\chi u_k\|_{L^p(M)} \le \|\chi \tilde u_k\|_{L^p(M)} + \mathrm{vol}(M)^{-1}\|\chi\|_{L^p}. 
\]
The second term on the right hand side is bounded so it is enough to bound the first term. 
\begin{equation*}
\begin{split}
\|\chi \tilde u_k\|_{L^p(M)} &\le \sum\limits_{m=0}^{\infty} \|Q_m(\chi \tilde u_k)\|_{L^p(M)} \\
&\le C_{M,p}\sum\limits_{m=0}^{\infty}(1+m)^{\delta(p)}\|Q_m(\chi \tilde u_k)\|_{L^2(M)}\\
&\le C_{M,N,p,\chi,x} \sum\limits_{m=0}^{\infty}(1+m)^{\delta(p)}(1+k)^{1/2}(1+|m-k|)^{-N}\\
&\le C_{M,N,p,\chi,x} \sum\limits_{m=0}^{\infty}(1+k)^{\delta(p)+1/2}(1+|m-k|)^{\delta(p)-N}\\
&\le C_{M,p,\chi,x} k^{\delta(p)+1/2},
\end{split}
\end{equation*}
where to go from the first to the second line we used the spectral cluster estimate \ref{Sogge}, to go from the second to the third line we used Lemma \ref{aux2}, to go from the third to the fourth line we used the simple inequality $(1+m)\le (1+k)(1+|m-k|)$ and to go from the forth to the fifth line we chose $N>\delta(p)+1$. 
\end{proof}
\subsection{The dyadic maximal estimate}

\begin{proposition}
For all $2<p<\infty$, $k\geq 2$ and $\chi$ as above, 
\[
\left\|\sup\limits_{k\le \lambda<k+1}|\chi E_{\lambda}(x,\cdot)|\right\|_{L^p(M)} \le C_{M,p,\chi,x} k^{1/2+\delta(p)}. 
\]
\end{proposition}
\begin{proof}
Write 
\[
\chi(y)E_{\lambda}(x,y) = \chi u_k + \chi(y)\sum\limits_{k\le \lambda_j <\lambda} \phi_j(x)\overline{\phi_j(y)}.
\]
Thus,
\begin{equation*}
\begin{split}
\left\|\sup\limits_{k\le \lambda<k+1}|\chi E_{\lambda}(x,\cdot)|\right\|_{L^p(M)} &\le \|\chi u_k\|_{L^p(M)} + \|\chi\|_{\infty} \left\|\sup\limits_{k\le \lambda<k+1}\left|\sum\limits_{k\le\lambda_j<\lambda}\phi_j(x)\overline{\phi_j(y)}\right|\right\|_{L^p(M)}\\
&\le C_{M,p,\chi,x} k^{1/2+\delta(p)} + C_{M,p,\chi,x} k^{\delta(p)}\left(\sum\limits_{k\le_j \lambda<k+1} |\phi_j(x)|^2\right)^{1/2}\\
&\le C_{M,p,\chi,x}k^{1/2+\delta(p)},
\end{split}
\end{equation*}
where to go from the first to the second line we used \ref{aux3} for the first term and \ref{CKC} for the second and to go from the second to the third we used the local Weyl law. 
\end{proof}
The following is an immediate corollary since the interval $[L,2L]$ can be covered by $O(L)$ unit intervals,
\begin{corollary}
Let $L\geq 2$ and $2<p<\infty$. Then 
\[
\left\|\sup\limits_{L\le \lambda<2L}|\chi E_{\lambda}(x,\cdot)|\right\|_{L^p(M)}\le C_{M,p,\chi,x}L^{1/2+\delta(p)+1/p}. 
\]
\end{corollary}
The optimal bound obtained by the method below is for $p=6$ for which the corollary gives 
\begin{equation}\label{mest1}
\left\|\sup\limits_{L\le \lambda<2L}|\chi E_{\lambda}(x,\cdot)|\right\|_{L^6(M)}\le C_{M,\chi,x}L^{5/6}.
\end{equation}

\subsection{Proof of \ref{M2}}

For each integer $l\geq 1$ choose $\chi_l$ which vanishes near $x$ and $\chi_l=1$ on $M\backslash B(x,1/l)$. Fix $l$ and define 
\begin{equation*}
\begin{split}
b_m &= m^{1/6}(\log(m))^{1/6+\eps}\\
\mathcal E_m^l &= \left\{y: \sup\limits_{2^m\le \lambda <2^{m+1}}|\chi_l(y)E_{\lambda(x,y)}|> 2^{5m/6}b_m\right\}
\end{split}
\end{equation*}
Inequality \eqref{mest1} and Chebyshev's inequality give
\[
\mu(\mathcal E_m^l) \le C_{M,\chi,x,l} b_m^{-6} = \frac{C_{M,\chi,x,l}}{m(\log(m))^{1+6\eps}}.
\]
This in particular implies that 
\[
\sum\limits_{m\geq 0} \mu(\mathcal E_m) <\infty,
\]
so the first Borel-Cantelli Lemma implies that 
\[
\mu(\limsup\limits_{m\to\infty} \mathcal E_m^l) = 0.
\]
Define $Z_x^l = M\backslash \limsup\limits_{m\to\infty} \mathcal E_m^l$ and 
\[
Z_x := \bigcap\limits_{l\geq 1} Z_x^l. 
\]
Then for any $y\in Z_x$ there exists $l$ such that $y\in M\backslash B(x,1/l)$ and therefore some $m_0$ such that for all $m\geq m_0$ and $\Lambda\le 2^m$
\[
|E_{\Lambda}(x,y)| \le 2^{5m/6}m^{1/6}(\log(m))^{1/6+\eps}. 
\]
The result follows. 

\bibliography{bibliography}
\bibliographystyle{amsalpha}

\vspace{12pt}
\noindent
Panagiotis Dimakis, Department of Mathematics,University of Maryland\\
College Park 20740, MD, USA.\\
\textit{pdimakis12345@gmail.com}

\end{document}